\documentclass[a4paper,11pt]{article}
\usepackage[top=30mm,left=29mm,right=29mm]{geometry}
\usepackage[round]{natbib}
\usepackage{hyperref}
\newcommand{\lcucolor}{0.1,0.1,0.9}
\hypersetup{colorlinks,
            linktocpage=true,
            linkcolor=[rgb]{\lcucolor},
            citecolor=[rgb]{\lcucolor},
            urlcolor=[rgb]{\lcucolor}}

\usepackage{amsmath}
\usepackage{amssymb}
\usepackage{amsthm}
\usepackage{enumitem}
\usepackage{xspace}
\usepackage{bm} 

\usepackage[capitalize]{cleveref}
\crefname{figure}{Figure}{Figures}

\usepackage[plain]{algorithm}
\usepackage{algorithmic}
\algsetup{linenodelimiter=.}

\usepackage[disable]{todonotes}
\newtheorem{theorem}{Theorem}
\newtheorem{lemma}[theorem]{Lemma}

\newcommand{\todob}[2][]{\todo[color=blue!20!white,#1]{BG: #2}}

\newcommand{\setN}{\mathbb{N}} 
\newcommand{\setR}{\mathbb{R}} 
\newcommand{\setX}{\mathbb{X}}
\newcommand{\setY}{\mathbb{Y}}
\newcommand{\graph}{\textrm{graph}} 
\newcommand{\cvxhull}{\textrm{conv}} 

\newcommand{\defeq}{\doteq} 
\newcommand{\T}{\top} 
\newcommand{\argmin}{\mathop{\mathrm{argmin}}}

\newcommand{\Ordo}{O} 

\newcommand{\norm}[1]{\left\lVert#1\right\rVert}
\newcommand{\inorm}[1]{\norm{#1}_\infty} 

\newcommand{\E}{\mathbb{E}} 
\newcommand{\median}{\textrm{median}} 

\renewcommand{\vec}[1]{{\boldsymbol{#1}}}
\newcommand{\vzero}{\vec{0}} 
\newcommand{\vx}{\vec{x}}
\newcommand{\vy}{\vec{y}}
\newcommand{\vz}{\vec{z}}
\newcommand{\va}{\vec{a}}
\newcommand{\vb}{\vec{b}}
\newcommand{\vc}{\vec{c}}

\newcommand{\vu}{\vec{u}}
\newcommand{\vh}{\vec{h}}
\newcommand{\vk}{\vec{k}}
\newcommand{\vr}{\vec{r}}

\newcommand{\Id}{I}

\newcommand{\vX}{\vec{\mathcal{X}}}
\newcommand{\X}{\mathcal{X}}
\newcommand{\Y}{\mathcal{Y}}
\newcommand{\Z}{\mathcal{Z}}
\newcommand{\vD}{\vec{D}}
\newcommand{\vZ}{\vec{\Z}}

\newcommand{\pol}{\pi}
\newcommand{\hpol}{\hat{\pol}}
\newcommand{\setXh}{\hat\setX}

\newcommand{\J}{J}
\newcommand{\Jh}{\hat{J}}

\newcommand{\INPUT}{\STATE \textbf{input:}\,\,}
\newcommand{\ASSIGN}{\leftarrow}
\newcommand{\OUTPUT}{\STATE \textbf{output:}\,\,}

\newcommand{\Dn}{\mathcal{D}_n} 
\newcommand{\model}{M} 
\newcommand{\partn}{P} 
\newcommand{\cell}{\mathcal{C}} 
\newcommand{\mcs}{s_*} 
\newcommand{\tfinal}{t_{\textrm{wait}}} 
\newcommand{\Cle}{\cell_{\textrm{le}}}
\newcommand{\Cgt}{\cell_{\textrm{gt}}}
\newcommand{\vale}{\va_{\textrm{le}}}
\newcommand{\ble}{b_{\textrm{le}}}
\newcommand{\vagt}{\va_{\textrm{gt}}}
\newcommand{\bgt}{b_{\textrm{gt}}}
\newcommand{\riskn}{R_n} 
\newcommand{\ma}[2]{f_{#1}(#2)} 
\newcommand{\Fbar}{\mkern 4mu\overline{\mkern-4mu F\mkern-0mu}\mkern-3mu}
\newcommand{\Ebar}{\mkern 4mu\overline{\mkern-4mu E\mkern-0mu}\mkern-1mu}
\newcommand{\setM}{\mathcal{M}}
\newcommand{\fold}{u} 
\newcommand{\nfolds}{\fold_*} 
\newcommand{\tmax}{t_{\max}}

\newcommand{\fes}{f_{\textrm{es}}}
\newcommand{\fed}{f_{\textrm{ed}}}
\newcommand{\feg}{f_{\textrm{eg}}}
\newcommand{\fsd}{f_{\textrm{sd}}}
\newcommand{\fgs}{f_{\textrm{gs}}}
\newcommand{\fsg}{f_{\textrm{sg}}}
\newcommand{\fest}{f_{\textrm{es},t}}
\newcommand{\fedt}{f_{\textrm{ed},t}}
\newcommand{\fegt}{f_{\textrm{eg},t}}
\newcommand{\fsdt}{f_{\textrm{sd},t}}
\newcommand{\fgst}{f_{\textrm{gs},t}}
\newcommand{\fsgt}{f_{\textrm{sg},t}}

\begin{document}

\title{Max-affine estimators for convex stochastic programming}

\author{G\'abor Bal\'azs and Csaba Szepesv\'ari \\
        Department of Computing Science,
        University of Alberta, \\
        Edmonton, Alberta, T6G 2E8, Canada \\
        \ \\
        Andr\'as Gy\"orgy \\
        Department of Electrical and Electronic Engineering, \\
        Imperial College London,
        London, SW7 2BT, United Kingdom}

\maketitle

\begin{abstract}
  In this paper, we consider two
  sequential decision making problems with a convexity structure,
  namely an energy storage optimization task
  and a multi-product assembly example.
  We formulate these problems in the stochastic programming framework and
  discuss an approximate dynamic programming technique for their solutions.
  As the cost-to-go functions are convex in these cases,
  we use max-affine estimates for their approximations.
  To train such a max-affine estimate,
  we provide a new convex regression algorithm,
  and evaluate it empirically for these planning scenarios.
\end{abstract}

\section{Introdution}

This paper considers multi-stage stochastic programming problems
\citep[e.g.,][]{ShapiroDentchevaRuszczynski2009,
                BirgeLouveaux2011}
with restricting the model so that the cost-to-go functions remain convex
(\cref{sec:stoprg}).
To motivate this framework,
we provide two realistic benchmark planning problems (\cref{sec:experi}),
namely solar energy production with storage,
and the operation of a beer brewery.

To address these problems,
we consider an approximate dynamic programming approach
\citep[e.g.,][]{Bertsekas2005,Powell2011}.
For this, we estimate the cost-to-go functions
by convex regression techniques using max-affine representations
(formed by the maximum of finitely many affine functions).
To train a max-affine estimator,
we propose a novel algorithm (AMAP, \cref{sec:amap}),
which combines ideas of
the convex adaptive partitioning method (CAP, \citealp{HannahDunson2013})
and the least squares partition algorithm (LSPA, \citealp{MagnaniBoyd2009}),
while learns the model size by cross-validation.

We discuss a full approximate dynamic programming approach (\cref{sec:fADP})
that estimates the cost-to-go functions globally. It uses a forward pass
to generate a data set for the cost-to-go estimation
by uniformly sampling the reachable decision space.
Then it performs a backward pass to recursively approximate the
cost-to-go functions for all stages.
The technical details of this algorithm is provided
for polyhedral decision sets
and convex piecewise-linear cost-to-go representations.

Finally, we evaluate max-affine estimators
in the contexts of our benchmark planning problems
through the full approximate dynamic programming algorithm
(\cref{sec:experi}).

\section{Convex stochastic programming}
\label{sec:stoprg}

Consider a $T$-stage stochastic programming (SP) problem
\citep[e.g.,][]{RuszczynskiShapiro2003,
                ShapiroDentchevaRuszczynski2009,
                BirgeLouveaux2011},
where the goal is to find a decision $\vx_1^*$ solving the following:
\begin{equation} \begin{split}
  \label{eq:cvx-sto-prg}
    \vx_1^* \in
    \argmin_{\vx_1 \in \setX_1(\vx_0,\vz_0)} &\J_1(\vx_1)
    \,, \quad
    \J_t(\vx_t) \defeq
    \E\Big[c_{t}(\vx_{t},\vZ_t)
           + \hspace{-1mm}
             \min_{\vx_{t+1} \in \setX_{t+1}(\vx_t,\vZ_t)}
             \hspace{-1mm}
             \J_{t+1}(\vx_{t+1})\Big]
    \,,
  \end{split} \end{equation}
  with $t = 1,\ldots,T$,
       some fixed initial values $\vx_0,\vz_0$,
       $\setX_{T+1}(\vx_T,\vZ_T) \defeq \{\bot\}$,
       $\J_{T+1}(\bot) \defeq 0$,
   and a sequence of independent random variables $\vZ_1,\ldots,\vZ_T$.

Notice that \eqref{eq:cvx-sto-prg} includes
discrete-time finite-horizon Markov decision process formulations
\citep[e.g.,][]{Puterman1994,Sutton1998,Bertsekas2005,
                Szepesvari2010,Powell2011}
after the state and action variables are combined
into a single decision variable $\vx_t$,
and the environment dynamics along with the action constraints are described
by the decision constraint functions $\setX_t$.

In this text, we consider only a subset of SP problems \eqref{eq:cvx-sto-prg}
when the cost functions $c_1,\ldots,c_T$ are convex in $\vx_t$,
and $\graph(\setX_t(\vx_t,\vZ_t))$ are convex sets
for all $t = 1,\ldots,T$ and all $\vZ_t$ realizations,
where the graph of a set-valued function $C : \setX \to 2^\setY$
is defined as $\graph(C) \defeq \big\{(\vx,\vy) \in \setX\times\setY :
                                      \vy \in C(\vx)\big\}$.
In this case, the cost-to-go functions $\J_t(\cdot)$
are convex for all $t = 1,\ldots,T$ (e.g., see \cref{thm:graph-inf}),
hence we call these SP problems convex.

One approach to deal with such SP problems \eqref{eq:cvx-sto-prg}
is to use approximate dynamic programming (ADP) methods
\citep[e.g.,][]{Bertsekas2005,Powell2011,
                BirgeLouveaux2011,HannahPowellDunson2014}
which construct nested approximations to the cost-to-go functions,
\begin{equation}
\label{eq:Japx}
  \Jh_t(\vx_t)
  \approx \E\left[c_t(\vx_t,\vZ_t) +
                  \min_{\vx_{t+1}\in\setX_{t+1}(\vx_t,\vZ_t)}
                  \Jh_{t+1}(\vx_{t+1})\right]
  \approx \J_t(\vx_t)
  \,,
\end{equation}
backwards for $t = T, T-1, \ldots, 1$.
Notice that for convex SP problems with convex \mbox{cost-to-go} functions $\J_t$,
the estimates $\Jh_t$ can be restricted to be convex without loss of generality,
and so the minimization tasks in \eqref{eq:Japx} can be solved efficiently.

\section{Adaptive max-affine partitioning algorithm}
\label{sec:amap}

To represent convex cost-to-go functions, we use max-affine maps
formed by the maximum of finitely many affine functions (hyperplanes).
To train the parameters of such a max-affine estimate,
we present a new convex regression algorithm
called \emph{Adaptive Max-Affine Partitioning} (AMAP),
which combines the partitioning technique of
the convex adaptive partitioning algorithm
(CAP, \citealp{HannahDunson2013}),
the least squared partition algorithm
(LSPA, \citealp{MagnaniBoyd2009}),
and learns the model size (number of hyperplanes)
by a cross-validation scheme.

Just as LSPA and CAP, AMAP also aims to reduce the empirical risk
with respect to the squared loss defined as
$\riskn(f) \defeq \frac1n\sum_{i=1}^n|f(\vx_i)-y_i|^2$
for some function $f$ and data set
$\Dn \defeq \{(\vx_i,y_i):i=1,\ldots,n\} \subseteq (\setR^d \times \setR)^n$
with $n$ samples and dimension $d$.
For the discussion of AMAP, denote the max-affine function
of model $\model = \{(\va_k,b_k):k=1,\ldots,K\}$ by
$\ma{\model}{\vx} \defeq \max_{k=1,\ldots,K}\va_k^\T\vx+b_k$
for $\vx \in \setR^d$.
Notice that each max-affine model $\model$ induces a partition
over the data set $\Dn$ as
\begin{equation} \begin{split}
  \label{eq:MindP}
  \partn &\defeq \{\cell_1,\ldots,\cell_K\}
  \,, \quad
  \cell_k  \defeq \big\{ i \in \{1,\ldots,n\} \,\big|\,
                         \vx_i^\T\va_k+b_k = \ma{\model}{\vx_i}\big\}
  \,,
\end{split} \end{equation}
for some $K \in \setN$ and all $k = 1,\ldots,K$,
where ties are broken arbitrarily so that
$\cell_k \ne \emptyset$
and $\cell_k \cap \cell_l \ne \emptyset \iff k = l$ for all $k,l = 1,\ldots,K$.
Furthermore, each partition $\partn$ induces a max-affine model
by fitting each cell of $\partn$
using the linear least squares algorithm:
\begin{equation} \begin{split}
  \label{eq:cwls}
  \model &\defeq \{(\va_k,b_k) : k = 1,\ldots,K\}
  \,, \quad
  \Delta_{ik} \defeq \vx_i - \frac{1}{|\cell_k|}\sum_{i\in\cell_k}\vx_i
  \,, \\
  \va_k &\defeq \Big(\sum_{i\in\cell_k}
                     \Delta_{ik}\Delta_{ik}^\T + \beta\Id_d\Big)^{-1}
                \sum_{i\in\cell_k}\Delta_{ik} \, y_i
  \,, \quad
  b_k \defeq \frac{1}{|\cell_k|}\sum_{i\in\cell_k}y_i-\va_k^\T\vx_i
  \,,
\end{split} \end{equation}
where $\Id_d$ is the $d \times d$ identity matrix,
and $\beta > 0$ is set to some small value for stability
(we use $\beta \defeq 10^{-6}$).

Similar to CAP, AMAP builds the model incrementally by cell splitting,
and improves the partition using LSPA by alternating steps
\eqref{eq:MindP} and \eqref{eq:cwls}.
The AMAP model improvement step is given by \cref{alg:AMAPstep}.
\begin{algorithm}[h!]
\centering
\begin{minipage}{0.9\columnwidth}
\begin{algorithmic}[1]
  \INPUT training set $\Dn$,
         model $\model = \{(\va_k,b_k) : k = 1,\ldots,K\}$,
  \item[] \hspace{12.5mm} partition $\partn = \{\cell_1,\ldots,\cell_K\}$,
          empirical risk $E$, minimum cell size $\mcs$
  \item[] \COMMENT{cell splitting}
  \STATE \label{eq:AMAPstep:csplit1}
         $\model_*' \ASSIGN \model, E_*' \ASSIGN E$,
         $\partn_*' \ASSIGN \partn$
  \FORALL{$k = 1,\ldots,K$}
  \IF{$|\cell_k| \ge 2\mcs$}
      \FORALL{$j = 1,\ldots,d$} \label{eq:AMAPstep:split1}
        \STATE \label{eq:AMAPstep:median1}
               $m_j \ASSIGN \median(\{\X_{ij} : i\in\cell_k\})$
        \STATE \label{eq:AMAPstep:median2}
               $\Cle \ASSIGN \{i\in\cell_k : \X_{ij} \le m_j\}$,
               $\Cgt \ASSIGN \{i\in\cell_k : \X_{ij} \ge m_j\}$
        \STATE \label{eq:AMAPstep:ridge1}
               $(\vale,\ble) \ASSIGN $ least squares
               on $\{(\vX_i,\Y_i) : i\in \Cle\}$
        \STATE \label{eq:AMAPstep:ridge2}
               $(\vagt,\bgt) \ASSIGN $ least squares
               on $\{(\vX_i,\Y_i) : i\in \Cgt\}$
        \STATE \label{eq:AMAPstep:model}
               $\model' \ASSIGN \big(\model \setminus \{(\va_k,b_k)\}\big)
                                      \cup \{(\vale,\ble), (\vagt,\bgt)\}$
        \STATE \label{eq:AMAPstep:risk}
               $E' \ASSIGN \riskn\big(\ma{M'}{\cdot}\big)$
        \IF{$E' < E_*'$}
          \STATE $\model_*' \ASSIGN \model'$,
                 $E_*' \ASSIGN E'$,
                 $\partn_*' \ASSIGN (\partn \setminus \{\cell_k\})
                                    \cup \{\Cle, \Cgt\}$
        \ENDIF
      \ENDFOR \label{eq:AMAPstep:split2}
    \ENDIF
  \ENDFOR \label{eq:AMAPstep:csplit2}
  \item[] \COMMENT{running LSPA}
  \REPEAT \label{eq:AMAPstep:LSPArep}
    \STATE \label{eq:AMAPstep:LSPAstart}
           $\model_* \ASSIGN \model_*'$,
           $E_* \ASSIGN E_*'$,
           $\partn_* \ASSIGN \partn_*'$
    \STATE $\partn_*' \ASSIGN $
           induced partition of $\model_*$ by \eqref{eq:MindP}
    \STATE \label{eq:AMAPstep:LSPAfit}
           $\model_*' \ASSIGN $ fitting of partition $\partn_*'$
           by \eqref{eq:cwls}
    \STATE \label{eq:AMAPstep:LSPAerror}
           $E_*' \ASSIGN \riskn\big(\ma{\model'_*}{\cdot}\big)$
  \UNTIL{$\min_{\cell\in\partn_*'}|\cell| \ge \mcs$
         and $E_*' < E_*$} \label{eq:AMAPstep:LSPAuntil}
  \OUTPUT model $\model_*$, partition $\partn_*$, empirical risk $E_*$
\end{algorithmic}
\end{minipage}
\caption{AMAP model improvement step.}
\label{alg:AMAPstep}
\end{algorithm}
Notice that AMAP performs coordinate-wise cell splitting
(steps~\ref{eq:AMAPstep:split1} to~\ref{eq:AMAPstep:split2}),
just as CAP, however, AMAP makes
the split always at the median (steps~\ref{eq:AMAPstep:median1} and
                                      \ref{eq:AMAPstep:median2})
instead of checking multiple cut points.
This saves computation time, but can also create worse splits.
To compensate for this loss in quality,
AMAP runs a restricted version of LSPA
(steps~\ref{eq:AMAPstep:LSPArep} to \ref{eq:AMAPstep:LSPAuntil})
not just for a single step as CAP,
but until the candidate model improves the empirical risk
and its induced partition satisfies
the minimum cell requirement (step~\ref{eq:AMAPstep:LSPAuntil}).
We also mention that indices $\{i\in\cell_k : \X_{ij} = m_j\}$
are assigned to $\Cle$ and $\Cgt$ (step~\ref{eq:AMAPstep:median2})
in order to preserve the minimum cell requirement.

Notice that the difference between $\model'$ and $\model$
is only two hyperplanes (step~\ref{eq:AMAPstep:model}),
so the number of arithmetic operations
for computing $E'$ (step~\ref{eq:AMAPstep:risk})
can be improved from $\Ordo(nKd)$ to $\Ordo(nd)$.
Further, the cost of least squares regressions
(steps \ref{eq:AMAPstep:ridge1} and \ref{eq:AMAPstep:ridge2})
is $\Ordo(|\cell_k|d^2)$.
Hence, the computational cost of the entire cell splitting process
(steps \ref{eq:AMAPstep:csplit1} to \ref{eq:AMAPstep:csplit2})
is bounded by $\Ordo(\max\{K,d\} d^2 n)$.
For the LSPA part, the partition fitting (step~\ref{eq:AMAPstep:LSPAfit})
is $\Ordo(n d^2)$ and the error calculation (step~\ref{eq:AMAPstep:LSPAerror})
is $\Ordo(nKd)$. So, the cost of a single LSPA iteration
(steps~\ref{eq:AMAPstep:LSPAstart} to~\ref{eq:AMAPstep:LSPAerror})
is bounded by $\Ordo(\max\{K,d\}d n)$, implying that the cost of
\cref{alg:AMAPstep} is bounded by
$\Ordo\big(\max\{t_{\textrm{LSPA}},d\}\max\{K,d\}dn\big)$,
where $t_{\textrm{LSPA}}$ denotes the number of LSPA iterations.

Undesirably, coordinate-wise cell splitting
does not provide rotational invariance.
To fix this, we run AMAP after a pre-processing step, which
uses thin singular value decomposition (thin-SVD)
to drop redundant coordinates
and align the data along a rotation invariant basis.
Formally, let the raw (but already centered) data be organized into
$X \in \setR^{n \times d}$ and $\vy \in \setR^n$.
Then, we scale the values
$[y_1 \ldots y_n]^\T \defeq \vy/\max\{1,\inorm{\vy}\}$,
and decompose $X$ by thin-SVD as $X = U S V^\T$, where
$U \in \setR^{n \times d}$ is semi-orthogonal,
$S \in \setR^{d \times d}$ is diagonal
with singular values in decreasing order,
and $V \in \setR^{d \times d}$ is orthogonal.
Coordinates related to zero singular values are dropped%
\footnote{By removing columns of $U$ and $V$, and columns and rows of $S$.}
and the points are scaled by $S$
as $[\vx_1\ldots\vx_n]^\T \defeq US/\max\{1,S_{11}\}$,
where $S_{11}$ is the largest singular value.
Now observe that rotating the raw points $X$ as $XQ$
with some orthogonal matrix $Q \in \setR^{d \times d}$
only transforms $V$ to $Q^\T V$
and does not affect the pre-processed points $\vx_1,\ldots,\vx_n$.
Finally, we note that thin-SVD can be computed using $\Ordo(n d^2)$
arithmetic operations (with $n \ge d)$,%
\footnote{First decompose $X$ by a thin-QR algorithm
          in $\Ordo(nd^2)$ time \citep[Section~5.2.8]{GolubLoan1996}
          as $X = QR$, where $Q \in \setR^{n \times d}$ has orthogonal columns
          and $R \in \setR^{d \times d}$ is upper triangular.
          Then apply SVD for $R$ in $\Ordo(d^3)$ time
          \citep[Section~5.4.5]{GolubLoan1996}.}
which is even less than the asymptotic cost of \cref{alg:AMAPstep}.

AMAP is presented as \cref{alg:AMAP}
and run using uniformly shuffled (and pre-processed) data $\Dn$,
and a partition $\{F_1,\ldots,F_{\nfolds}\}$ of $\{1,\ldots,n\}$
with equally sized cells (the last one might be smaller)
defining the cross-validation folds.
\begin{algorithm}[h]
\centering
\begin{minipage}{0.98\columnwidth}
\begin{algorithmic}[1]
  \INPUT training set $\Dn$, minimum cell size $\mcs$,
         folds $F_1,\ldots,F_{\nfolds}$, iterations $\tfinal$
  \item[] \COMMENT{initialization}
  \FOR{$\fold = 1,\ldots,\nfolds$} \label{eq:AMAP:init1}
    \STATE $\partn_{\fold} \ASSIGN \{\Fbar_{\fold}\}$
           with $\Fbar_{\fold} \defeq \{1,\ldots,n\} \setminus F_{\fold}$
    \STATE $\model_{\fold} = \{(\va_1^{\fold},b_1^{\fold})\} \ASSIGN $
           least squares on $\{(\vx_i,y_i) : i \in \Fbar_{\fold}\}$
    \STATE $\Ebar_{\fold} \ASSIGN |\Fbar_{\fold}|^{-1}
                                  \sum_{i\in\Fbar_{\fold}}
                                  |\ma{\model_{\fold}}{\vx_i}-y_i|^2$,
           $E_{\fold} \ASSIGN |F_{\fold}|^{-1}
                              \sum_{i\in F_{\fold}}
                              |\ma{\model_{\fold}}{\vx_i}-y_i|^2$
  \ENDFOR \label{eq:AMAP:init2}
  \STATE \label{eq:AMAP:best1}
         $\setM_* \ASSIGN \{\model_1,\ldots,\model_{\nfolds}\}$,
         $E_* \ASSIGN \frac{1}{\nfolds}\sum_{\fold=1}^{\nfolds}E_{\fold}$
  \item[] \COMMENT{cross-validation training}
  \STATE \label{eq:AMAP:cvstart}
         $t \ASSIGN 1$, $\tmax \ASSIGN \tfinal$
  \WHILE{$t \le \min\{\tmax,\lceil n^{d/(d+4)} \rceil\}$}
    \FOR{$\fold = 1,\ldots,\nfolds$}
      \STATE \label{eq:AMAP:update}
             $(\model_{\fold},\partn_{\fold},\Ebar_{\fold}) \ASSIGN $
             update by \cref{alg:AMAPstep}
             using $\{(\vx_i,y_i) : i \in \Fbar_{\fold}\}$,
                   $\model_{\fold}$, $\partn_{\fold}$,
                   $\Ebar_{\fold}$, $\mcs$
      \STATE \label{eq:AMAP:testerr}
             $E_{\fold} \ASSIGN |F_{\fold}|^{-1}
                                \sum_{i\in F_{\fold}}
                                |\ma{\model_{\fold}}{\vx_i}-y_i|^2$
    \ENDFOR
    \STATE \label{eq:AMAP:cverr}
           $E \ASSIGN \frac{1}{\nfolds}\sum_{\fold=1}^{\nfolds}E_{\fold}$
    \IF{$E < E_*$}
      \STATE \label{eq:AMAP:best2}
             $\setM_* \ASSIGN \{\model_1,\ldots,\model_{\nfolds}\}$,
             $E_* \ASSIGN E$,
             $\tmax \ASSIGN t + \tfinal$
    \ENDIF
    \STATE $t \ASSIGN t+1$
  \ENDWHILE \label{eq:AMAP:cvend}
  \item[] \COMMENT{choosing the final model}
  \STATE \label{eq:AMAP:fmodel}
         $\model_* \ASSIGN \argmin_{\model \in \setM_*}
                           \riskn\big(\ma{\model}{\cdot}\big)$
  \OUTPUT model $\model_*$
\end{algorithmic}
\end{minipage}
\caption{Adaptive max-affine partitioning (AMAP).}
\label{alg:AMAP}
\end{algorithm}

For model selection, AMAP uses a $\nfolds$-fold cross-validation scheme
(steps \ref{eq:AMAP:cvstart} to \ref{eq:AMAP:cvend})
to find an appropriate model size
that terminates when the cross-validation error (step~\ref{eq:AMAP:cverr})
of the best model set $\setM_*$
(steps~\ref{eq:AMAP:best1} and \ref{eq:AMAP:best2})
cannot be further improved for $\tfinal$ iterations.
At the end, the final model is chosen
from the model set $\setM_*$ having the best cross-validation error,
and selected so to minimize the empirical risk on the entire data
(step~\ref{eq:AMAP:fmodel}).
For this scheme, we use the parameters
$\nfolds \defeq 10$ and $\tfinal \defeq 5$.

AMAP starts with models having a single hyperplane
(steps~\ref{eq:AMAP:init1} to \ref{eq:AMAP:init2})
and increments each model by at most one hyperplane in every iteration
(step~\ref{eq:AMAP:update}).
Notice that if AMAP cannot find a split for a model $\model_{\fold}$
to improve the empirical risk~$\Ebar_{\fold}$,
the update for model $\model_{\fold}$
(steps~\ref{eq:AMAP:update} and \ref{eq:AMAP:testerr})
can be skipped in the subsequent iterations
as \cref{alg:AMAPstep} is deterministic.
We also mention that for the minimum cell size,
we use $\mcs \defeq \max\{2(d+1),\lceil \log_2(n)\rceil\}$
allowing model sizes up to $\Ordo\big(n^{d/(d+4)}/\ln(n)\big)$,
which is enough for near-optimal worst-case performance
\citep[Theorems~4.1 and~4.2]{BalazsGyorgySzepesvari2015}.

\section{Approximate dynamic programming}
\label{sec:fADP}

Here we use max-affine estimators to approximate the cost-to-go
functions of convex SP problems.
First, notice that solving \eqref{eq:cvx-sto-prg} is equivalent to
the computation of $\pol_1(\vx_0,\vz_0)$, where
\begin{equation} \begin{split}
  \pol_t(\vx_{t-1},\vz_{t-1})
    &\defeq \argmin_{\vx_t\in\setX_t(\vx_{t-1},\vz_{t-1})} \J_t(\vx_t)
  \,, \\
  \J_t(\vx_t) &= \E\big[c_t(\vx_t,\vZ_t)
               + \J_{t+1}(\pol_{t+1}(\vx_t,\vZ_t))\big]
  \,,
\end{split} \end{equation}
for all $t = 1,\ldots,T$,
and $\pol_{T+1}(\cdot,\cdot) \defeq \bot$ with $\J_{T+1}(\bot) = 0$.
The sequence $\pol \defeq (\pol_1,\ldots,\pol_T)$ represents an
optimal policy.

We only consider SP problems with convex polyhedral decision constraints
written as
$\setX_{t+1}(\vx_t,\vZ_t)
 = \{\vx_{t+1} : Q_{t+1}\vx_{t+1} + W_{t+1}(\vZ_t)\vx_t \le \vc_{t+1}(\vZ_t)\}$
which are non-empty for all possible realizations of $\vx_t$ and $\vZ_t$.
As the coefficient~$Q_{t+1}$ of the decision variable $\vx_{t+1}$ is independent
of random disturbances $\vZ_t$ and the constraint
$\vx_t \in \setX_t(\vx_{t-1},\vz_{t-1})$ for policy $\pol_t$ is feasible
for any $\vx_{t-1}$ and $\vz_{t-1}$,
these SP problems are said to have a fixed, relatively complete recourse
\citep[Section~2.1.3]{ShapiroDentchevaRuszczynski2009}.
We will exploit the fixed recourse property for sampling
\eqref{eq:sto-prog:spl}, while relatively complete recourse
allows us not to deal with infeasibility issues
which could make these problems very difficult to solve.%
\footnote{Infeasible constraints can be equivalently modeled by
          cost-to-go functions assigning infinite value
          for points outside of the feasible region.
          Then notice that the estimation of functions
          with infinite magnitude and slope can be arbitrarily hard
          even for convex functions
          \citep[Section~4.3]{BalazsGyorgySzepesvari2015}.}

In order to construct an approximation $\Jh_t$ to the cost-to-go function $\J_t$,
we need ``realizable'' decision samples $\vx_{t,i}$ at stage $t$.
We generate these incrementally during a forward pass for $t = 1,2,\ldots,T$,
where given $n$ decisions $\vx_{t,1},\ldots,\vx_{t,n}$
and $m$ disturbances $\vz_{t,1},\ldots,\vz_{t,m}$ at stage $t$,
we uniformly sample new decisions for stage $t+1$ from the set
\begin{equation} \begin{split}
\label{eq:sto-prog:spl}
  \setXh_{t+1} &\defeq
  \cvxhull\bigg(\bigcup_{i=1,\ldots,n, \atop j = 1,\ldots,m}
                \setX_{t+1}(\vx_{t,i},\vz_{t,j})\bigg)
                \\
               &=
  \Big\{\vx_{t+1} : Q_{t+1}\vx_{t+1} \le
                    \max_{i=1,\ldots,n,\atop j=1,\ldots,m}
                    \big\{\vc_{t+1}(\vz_{t,j})
                          - W_{t+1}(\vz_{t,j})\vx_{t,i}\big\}
  \Big\}
  \,,
\end{split} \end{equation}
where $\cvxhull(\cdot)$ refers to the convex hull of a set,
and the maximum is taken component-wise.
To uniformly sample the convex polytope $\setXh_{t+1}$,
we use the Hit-and-run Markov-Chain Monte-Carlo algorithm
(\citealp{Smith1984}, or see \citealp{Vempala2005})
by generating 100 chains (to reduce sample correlation)
each started from the average of 10 randomly generated border points,
and discarding $d_{t+1}^2$ samples on each chain during the burn-in phase,%
\footnote{The choice was inspired by the $\Ordo(d_{t+1}^2)$ mixing result
          of the Hit-and-run algorithm \citep{Lovasz1999}.}
where $d_{t+1}$ is the dimension of $\vx_{t+1}$.

Then, during a backward pass for $t = T,T-1,\ldots,1$,
we recursively use the cost-to-go estimate of the previous stage
$\Jh_{t+1}(\cdot)$
to approximate the values of the cost-to-go function~$\J_t$
at the decision samples $\vx_{t,i}$ generated during the forward pass,
that is
\begin{equation}
\label{eq:sto-prog:yt}
  \J_t(\vx_{t,i}) \approx y_{t,i}
  \defeq \frac{1}{m}\sum_{j=1}^m
         c_t(\vx_{t,i},\vz_{t,j})
       + \min_{\vx_{t+1} \in \setX_{t+1}(\vx_{t,i},\vz_{t,j})}
         \Jh_{t+1}(\vx_{t+1})
\end{equation}
for all $t = 1,\ldots,T$, and $\Jh_{T+1}(\cdot) \equiv 0$.
This allows us to set up a regression problem with data
$\{(\vx_{t,i},y_{t,i}) : i=1,\ldots,n\}$, and to construct
an estimate $\Jh_t(\cdot)$ of the cost-to-go $\J_t(\cdot)$.

We call the resulting method, shown as \cref{alg:fADP},
full approximate dynamic programming (fADP)
because it constructs global approximations to the cost-to-go functions.%
\begin{algorithm}[h]
\centering
\begin{minipage}{0.96\columnwidth}
\begin{algorithmic}[1]
  \INPUT SP problem, trajectory count $n$,
         evaluation count $m$, estimator REG
  \STATE $\vx_{0,i} \ASSIGN \vx_0$ for all $i = 1,\ldots,n$
  \STATE $\vz_{0,j} \ASSIGN \vz_0$ for all $j = 1,\ldots,m$
  \item[] \COMMENT{forward pass}
  \FORALL{$t = 0,1,\ldots,T-1$}
    \STATE sample $\vx_{t+1,1},\ldots,\vx_{t+1,n}$ from $\setXh_{t+1}$
           by \eqref{eq:sto-prog:spl}
           using $\vx_{t,1},\ldots,\vx_{t,n}$,
             and $\vz_{t,1},\ldots,\vz_{t,m}$
    \STATE sample $\vz_{t+1,1},\ldots,\vz_{t+1,m}$
           from the distribution of $\vZ_{t+1}$
  \ENDFOR
  \item[] \COMMENT{backward pass}
  \STATE $\Jh_{T+1}(\cdot) \ASSIGN 0$
  \FORALL{$t = T,T-1,\ldots,1$}
    \STATE compute $y_{t,1},\ldots,y_{t,n}$ by \eqref{eq:sto-prog:yt}
           using $\Jh_{t+1}(\cdot)$, $\vx_{t,1},\ldots,\vx_{t,n}$,
                                 and $\vz_{t,1},\ldots,\vz_{t,m}$
    \STATE $\Jh_t(\cdot) \ASSIGN
            \textrm{REG}(\{(\vx_{t,i},y_{t,i}) : i=1,\ldots,n\})$
  \ENDFOR
  \OUTPUT cost-to-go functions $\Jh_t(\cdot)$, $t = 1,\ldots,T$
\end{algorithmic}
\end{minipage}
\caption{Full approximate dynamic programming (fADP).}
\label{alg:fADP}
\end{algorithm}

When the cost-to-go functions $\J_t$ are approximated
by ``reasonably sized''
convex piecewise linear representations $\Jh_t$,
the minimization problem in \eqref{eq:sto-prog:yt}
can be solved efficiently by linear programming (LP).
In the following sections, we exploit the speed of LP solvers
for fADP using either AMAP or CAP as the regression procedure.
Then, the computational cost to run \cref{alg:fADP} is mostly realized
by solving $Tnm$ LP tasks for \eqref{eq:sto-prog:yt}
and training $T$ estimators using the regression algorithm REG.

Here we mention that max-affine estimators
using as many hyperplanes as the sample size $n$
can be computed by convex optimization techniques
efficiently up to a few thousand samples \citep{MazumderEtAl2015}.
These estimators provide worst-case generalization error guarantees,
however, the provided max-affine models are large enough to
turn the LP tasks together too costly to solve
(at least using our hardware and implementation).
\todob{Future work for tuning LP methods for solving large number
       of similar tasks. There is some literature on this
       (as mentioned at the end of the section),
       see \citet{GassmannWallace1996}.
       But commercial solvers do not provide much support.}

The situation is even worse for nonconvex estimators REG
for which LP has to be replaced for \eqref{eq:sto-prog:yt}
by a much slower nonlinear constrained optimization method
using perhaps randomized restarts to minimize the chance of being trapped
in a local minima.
Furthermore, when the gradient information with respect to the input
is not available (as for many implementations),
the minimization over these representations require
an even slower gradient-free nonlinear optimization technique.
Hence, multivariate adaptive regression splines,
support vector regression, and neural networks
were impractical to use for fADP on our test problems.

\section{Experiments}
\label{sec:experi}

To test fADP using max-affine estimators,
we consider two SP planning problems, namely
solar energy production with storage management (\cref{sec:prb:energy-opt}),
and the operation of a beer brewery (\cref{sec:prb:beer-brew}).

For our numerical experiments, the hardware has
a Dual-Core AMD Opteron(tm) Processor 250
(2.4GHz, 1KB L1 Cache, 1MB L2 Cache) with 8GB RAM.
The software uses MATLAB (R2010b),
and the MOSEK Optimization Toolbox (v7.139).

To measure the performance of the fADP algorithm,
we evaluate the greedy policy with respect to the learned
cost-to-go functions $\{\Jh_t : t = 1,\ldots,T\}$.
More precisely, we run $\hpol \defeq (\hpol_1,\ldots,\hpol_T)$ with
$\hpol_t(\vx_{t-1},\vz_{t-1})
 \in \argmin_{\vx_t\in\setX_t(\vx_{t-1},\vz_{t-1})}\Jh_t(\vx_t)$
on $1000$ episodes, and record the average revenue (negative cost) as
$\textrm{REV} \defeq -\frac{1}{1000}\sum_{e=1}^{1000}
                      \sum_{t=1}^T c_t\big(\vx^{(e)}_t,\vz^{(e)}_t\big)$
over the episodes' trajectories
$\big\{\big(\vx^{(e)}_t,\vz^{(e)}_t\big) : t = 1,\ldots,T\big\}$,
$e = 1,\ldots,1000$.
We repeat this experiment $100$ times
for each regression algorithm REG,%
\footnote{The random seeds are kept synchronized,
          so every algorithm is evaluated
          on the same set of trajectories.
          Furthermore, fADP algorithms with the same $n$ and $m$ parameters
          use the same training data $\vx_{t,1},\ldots,\vx_{t,n}$
          and $\vz_{t,1},\ldots,\vz_{t,m}$ for all $t = 0,\ldots,T$.}
and show the mean and standard deviation of the resulting sample.

\subsection{Energy storage optimization}
\label{sec:prb:energy-opt}

Inspired by a similar example of \citet[Section~7.3]{JiangPowell2015},
we consider an energy storage optimization problem where
a renewable energy company makes a decision every hour
and plans for two days ($T = 48$).
The company owns an energy storage with state~$s$
which can be charged with maximum rate $r_c$, using
the company's renewable energy source~($E$)
or the electrical grid that the company can buy electricity from
while paying the retail price~($p$).
The goal is to maximize profit by selling electricity to local clients
on retail price~($p$) according to their stochastic demand~($D$)
or selling it back to the electrical grid on wholesale price~($w$).
Electricity can be sold directly from the renewable energy source
or from the battery with maximum discharge rate $r_d$.
The energy flow control variables, $\fes, \fed, \feg, \fsd, \fsg, \fgs$,
are depicted on \cref{fig:energy-storage}.
\begin{figure}[h!] \begin{center}
  \includegraphics[width=9.5cm,height=4cm]{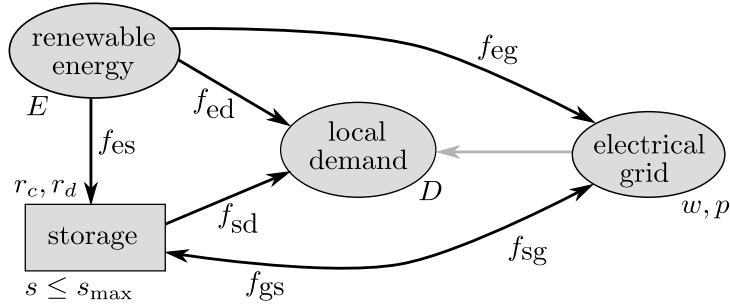}
  \caption{Flow diagram of a convex energy storage problem.}
  \label{fig:energy-storage}
\end{center} \end{figure}

The SP model \eqref{eq:cvx-sto-prg} of the energy storage problem
can be formulated by using the decision variable
\mbox{$\vx_t \defeq [s_t\,\,\fest\,\,\fedt\,\,\fegt
                        \,\,\fsdt\,\,\fsgt\,\,\fgst]^\T
       \in [0,\infty)^7$}
and setting $\vZ_t \defeq [E_{t+1}\,\,D_{t+1}]^\T$.
Furthermore, the cost function is defined as
\begin{equation*}
  c_t(\vx_t,\vZ_t) \defeq p_t(\fgst-\fedt-\fsdt) -w_t(\fegt+\fsgt)
  \,,
\end{equation*}
for all $t = 1,\ldots,T$,
and the dynamics and control constraints are described by
\begin{equation*}
  \setX_{t+1}(\vx_t,\vZ_t) \defeq
    \left\{\left[\begin{array}{c}
                 s \\ \fes \\ \fed \\ \feg \\ \fsd \\ \fsg \\ \fgs
                 \end{array}\right]
    \,\middle|\,
    \begin{array}{c}
      s = s_t + \fest - \fsdt - \fsgt + \fgst ,
      \\
      \fes, \fed, \feg, \fsd, \fsg, \fgs \ge 0 ,\,
      \\
      0 \le s + \fes - \fsd -\fsg + \fgs \le s_{\max} ,
      \\
      \fes + \fgs \le r_c ,\, \fsd + \fsg \le r_d ,
      \\
      \fes + \fed + \feg \le E_{t+1} ,\,
      \fed + \fsd \le D_{t+1}
    \end{array}
    \right\}
  \,,
\end{equation*}
for all $t = 0,\ldots,T-1$.
To initialize the system, define $\vx_0 \defeq [s_0\,\,0\ldots0]^\T$
and $\vz_0 \defeq [d_1\,\,e_1]^\T$,
where $s_1 = s_0 \in [0,s_{\max}]$ is the current storage level
and \mbox{$d_1,e_1 \ge 0$}
are the currently observed demand and energy production, respectively.

To set the parameters of this energy storage optimization problem,
we consider a solar energy source, a discounted nightly electricity pricing model
(Economy~7 tariff), and planning for a two days horizon on hourly basis
($T \defeq 48$).
Retail and expected wholesale price curves, along with
the electricity demand and energy production distributions
of this model are shown on \cref{fig:esp} for
the two consecutive sunny days.
\begin{figure}[h!] \begin{center}
  \input{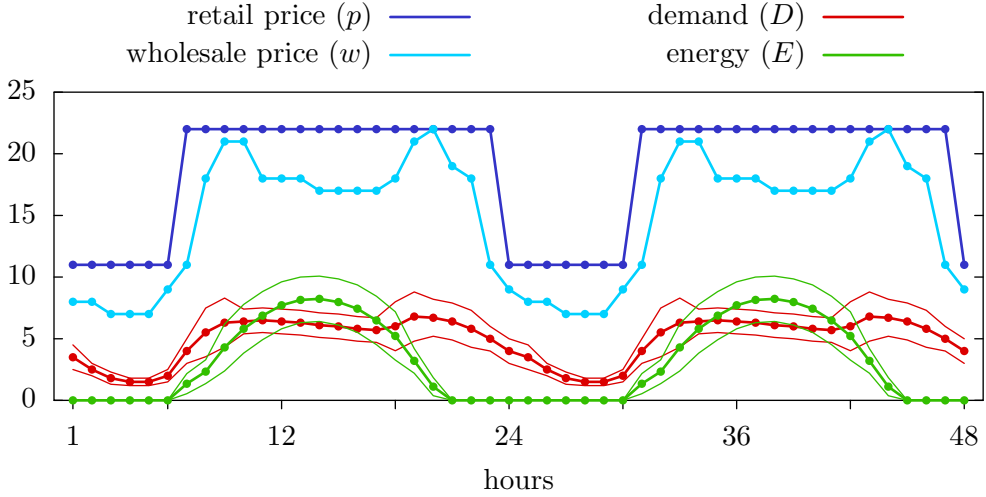}
  \vspace{2mm}
  \caption{Parameters of the energy storage optimization problem.
           Retail~($p$) and expected wholesale~($w$) price curves,
           energy demand~($D$) and production~($E$) distributions
           with mean and standard deviation
           are shown for two-day long period.}
  \label{fig:esp}
  \vspace{-1mm}
\end{center} \end{figure}
The distributions are truncated normal with support
$D \in [0,15]$ and $E \in [0,12]$
for demand and energy production, respectively
and the storage has capacity $s_{\max} \defeq 20$
with charge and discharge rates $r_c \defeq 4$ and $r_d \defeq 10$,
respectively.
The model is initialized by $s_0 \defeq 0$,
$d_1 \defeq \E[D_1]$, and $e_1 \defeq \E[E_1]$.

To evaluate fADP on this problem,
we use the CAP\footnote{As the implementation of CAP is not too reliable for
highly-correlated features, we combined it with the data preprocessing step
of AMAP (see \cref{sec:amap}) which improved its stability.}
and AMAP convex regression algorithms,
and multiple configurations determined by the number of trajectories $n$
generated for training (which is the sample size for the regression
tasks as well), and the number of evaluations $m$ used to approximate
the cost-to-go functions $\J_t$ at a single point \eqref{eq:sto-prog:yt}.
The result is presented on \cref{fig:result-plan-esp},
\begin{figure}[h!] \begin{center}
  \input{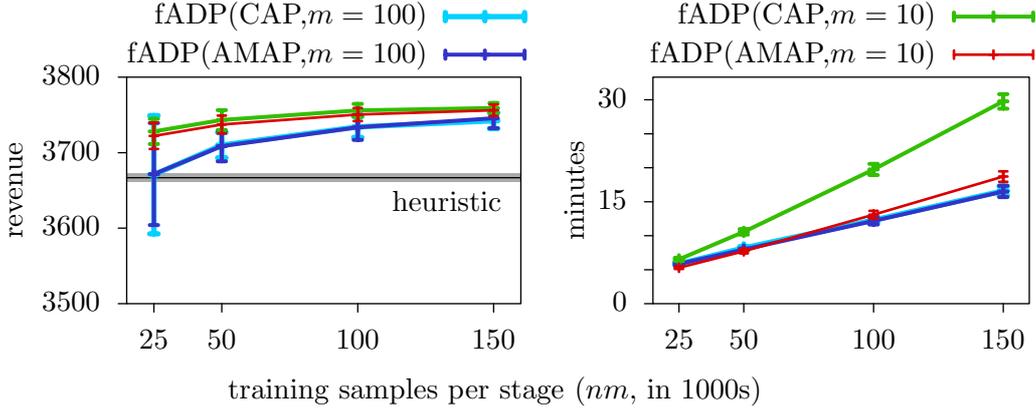}
  \caption{Energy storage optimization results for the fADP algorithm
           using AMAP or CAP as the inner convex regression procedure.
           Results show the total revenue and the training time in minutes
           with mean and standard deviation
           for the number of trajectories~$n$ and cost-to-go evaluations~$m$.}
  \label{fig:result-plan-esp}
  \vspace{-5mm}
\end{center} \end{figure}
which also includes a ``heuristic'' algorithm to provide a baseline.
The heuristic uses a fixed policy of immediately selling the solar energy
preferrably for demand ($\fed \to \max$, $\feg \ge 0$, $\fes = 0$),
selling from the battery during the day when demand still allows
($\fgs = 0$, $\fsd \ge 0$),
charging the battery overnight ($\fgs \to \max$, $\fsd = 0$),
and selling everything close to the last stage
($\fgs = 0$, $\fsd \to \max$, $\fsg \to \max$).
This policy is much better than the \emph{optimal policy without storage}%
\footnote{Because $p \ge w$,
          the optimal policy for $s_{\max} = 0$
          minimizes the immediate cost by
          $\fed \defeq \min\{E,D\}$,
          $\feg \defeq \max\{0,E-\fed\}$,
          and $\fes \defeq \fsd \defeq \fgs \defeq \fsg \defeq 0$.}
which scores $3227 \pm 6$.

The results of \cref{fig:result-plan-esp} show that fADP
using either CAP or AMAP significantly outperforms
the heuristic baseline algorithm when the sample size is large enough.
The regression algorithms prefer larger sample sizes $n$
to better sample quality $m$,
although this significantly increases the computation time for CAP
and provides only an insignificant revenue increase compared to AMAP.

\subsection{Beer brewery operation}
\label{sec:prb:beer-brew}

Inspired by \citet[Chapter~17]{AIMMS2016},
we consider the multi-product assembly problem of operating a beer brewery
which makes a decision in every two weeks and plans for about one year
($48$ weeks, $T = 24$).
The factory has to order ingredients (stratch source, yeast, hops)
to produce two types of beers (ale and lager) which have to be fermented
(for at least $2$ weeks for ale and $6$ weeks for lager) before selling.
The states and actions of this process
are illustrated on \cref{fig:beer-brewery}.
\begin{figure}[h!] \begin{center}
  \includegraphics[width=12cm,height=5cm]{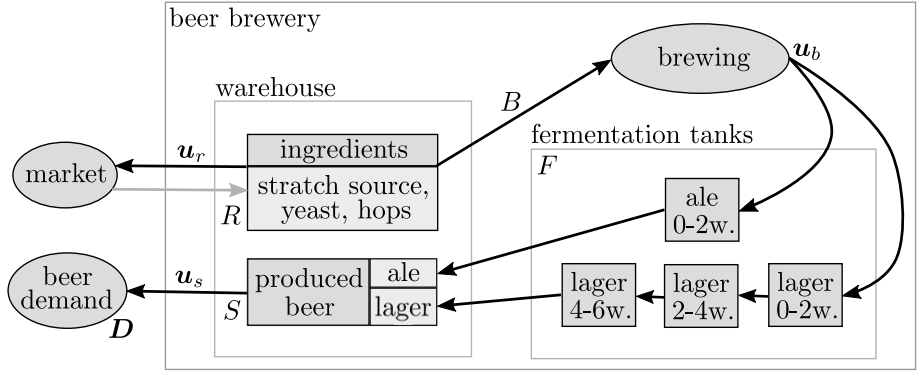}
  \caption{Diagram of the beer brewery problem.
           The state includes three types of ingredients and two types of
           bottled beer in the warehouse, and the four fermentation tanks.
           Actions are formed by ingredient orders $\vu_r$,
           brewing amounts~$\vu_b$, and beer sales $\vu_s$.}
  \label{fig:beer-brewery}
\end{center} \end{figure}

The decision variable $\vx_t$ is a $16$ dimensional vector with the
following components:
\begin{equation*}
  \vx_t \defeq \left[ \begin{array}{c}
               \textrm{stratch source in storage} \\
               \textrm{yeast in storage} \\
               \textrm{hops in storage} \\
               \textrm{ale beer fermented for less than 2 weeks} \\
               \textrm{produced ale beer} \\
               \textrm{lager beer fermented for less than 2 weeks} \\
               \textrm{lager beer fermented for 2 to 4 weeks} \\
               \textrm{lager beer fermented for 4 to 6 weeks} \\
               \textrm{produced lager beer} \\[1mm]
               \hline \\[-4mm]
               \textrm{stratch source order} \\
               \textrm{yeast order} \\
               \textrm{hops order} \\
               \textrm{ale beer brewing} \\
               \textrm{lager beer brewing} \\
               \textrm{ale beer sales} \\
               \textrm{lager beer sales} \\
               \end{array} \right]
  \in [0,\infty)^{16}
  \,.
\end{equation*}
The first $9$ coordinates are state variables,
while the last $7$ coordinates represent actions.
The cost functions (which may take negative values) are
$c_t(\vx_t,\vZ_t) \defeq [\vh_t^\T\,\vc_t^\T\,-\vr_t^\T]\vx_t$,
where $\vh_t \in [0,\infty)^{9}$ is the storage cost,
      $\vc_t \in [0,\infty)^{5}$ is the market price of the ingredients
with the brewing costs (adjusted by the water price),
and $\vr_t \in [0,\infty)^{2}$ is the selling price of the beers,
for stages $t = 1,\ldots,T$.
The constraint on the dynamics of the system is given by
\begin{equation*}
  \setX_{t+1}(\vx_t,\vZ_t) \defeq
    \left\{\left[\hspace{-1mm}\begin{array}{c}
                 F\vx_t + R\vu_r + B\vu_b - S\vu_s \\
                 \vu_r \\
                 \vu_b \\
                 \vu_s \\
           \end{array}\hspace{-1mm}\right]
           \middle|
           \begin{array}{c}
             \vu_r, \vu_b \ge \vzero ,\, \vu_s \in [\vzero,\vD_{t+1}], \\
             F\vx_t + B\vu_b - S\vu_s \ge \vzero, \\
             F\vx_t + R\vu_r + B\vu_b \le \vk_{t+1} \\
           \end{array}
    \right\}
  \,,
\end{equation*}
for all $t = 1,\ldots,T-1$, where
$\vZ_t = \vD_{t+1} \in [0,\infty)^2$
is the beer demand, $\vk_{t+1} \in [0,\infty]^9$
is the capacity bound,
$\vzero$ denotes a zero vector of appropriate size,
and
       the fermentation matrix $F \in \{0,1\}^{9\times16}$,
       the brewing matrix $B \in \setR^{9\times2}$,
       the storage loading matrix $R \in \{0,1\}^{9\times3}$
   and the selling matrix $S \in \{0,1\}^{9\times2}$
   are defined as
\begin{equation*}
  \arraycolsep=2pt
  F \defeq \left[
           \begin{array}{c|c|c}
              I_3 & \multicolumn{2}{c}{\vzero_{3\times13}}
              \\
              \hline
              \vzero_{6\times3}\,
              &
              \begin{array}{cc}
                \begin{array}{cc}
                  0 & 0 \\
                  1 & 1 \\
                \end{array}
                & \vzero_{2\times4}
                \\
                \vzero_{4\times2}
                & \begin{array}{cccc}
                    0 & 0 & 0 & 0 \\
                    1 & 0 & 0 & 0 \\
                    0 & 1 & 0 & 0 \\
                    0 & 0 & 1 & 1 \\
                  \end{array}
              \end{array}
              &
              \,\vzero_{6\times7}
           \end{array}
           \right]
\,, \quad
\begin{array}{c}
  B \defeq \left[\begin{array}{cc}
                 \hspace{-1mm}-\vb_a & \hspace{-1mm}-\vb_l\hspace{-1mm} \\
                 1 & 0 \\
                 0 & 0 \\
                 0 & 1 \\
                 \hline
                 \multicolumn{2}{c}{\vzero_{3\times2}} \\
           \end{array}\right] \,,
  \\
  \\
  R \defeq \left[\begin{array}{c}
              \Id_3 \\
              \vzero_{6\times3} \\
           \end{array}\right] \,,
  \end{array}
  \quad
  S \defeq \left[
             \begin{array}{c}
               \vzero_{3\times2} \\
               \hline
               \begin{array}{cc}
                 0 & 0 \\
                 1 & 0 \\
                 0 & 0 \\
                 0 & 0 \\
                 0 & 0 \\
                 0 & 1 \\
               \end{array}
             \end{array}
             \right]
    \,,
\end{equation*}
where $\vzero_{m \times n}$ is a zero matrix of size $m \times n$,
and $\vb_a, \vb_l \in [0,\infty)^3$
are the required ingredients for brewing ales and lagers, respectively.
To initialize the system,
we use $\vx_0 \defeq \vzero$ and $\vz_0 \defeq \vzero$.

To test the fADP algorithm on this problem, we set
the horizon to $48$ weeks horizon on a fortnight basis ($T \defeq 24$).
The demand distributions $\vZ_t$
for lager and ale beers are shown on \cref{fig:bb}
\begin{figure}[h!] \begin{center}
  \input{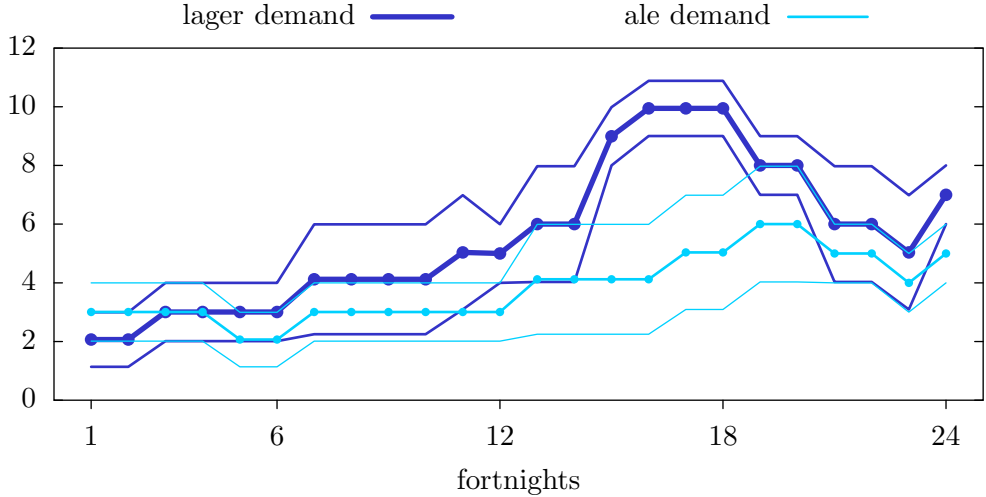}
  \vspace{2mm}
  \caption{Lager and ale beer demand distributions
           for the beer brewery optimization problem
           with mean and standard deviation are shown
           for a $48$ weeks horizon.}
  \label{fig:bb}
\end{center} \end{figure}
with mean and standard deviation.
Both distributions are truncated normal with support $[0.1,12]$.
The cost vectors are set to fixed values
for all $t = 1,\ldots,T$ as
$\vh_t \defeq [1\,\,0.2\,\,0.2\,\,1\,\,2\,\,1\,\,1\,\,1\,\,2]^\T$,
$\vc_t \defeq [20\,\,10\,\,5\,\,1\,\,1]^\T$,
and $\vr_t \defeq [90\,\,50]^\T$.
Furthermore, the ingredient requirement vectors for brewing are
$\vb_a \defeq [1\,\,1\,\,1]^\T$ and $\vb_l \defeq [0.5\,\,0.9\,\,0.8]^\T$
for ale and lager, respectively, and the capacity vector is
$\vk \defeq [10\,\,10\,\,10\,\,10\,\,\infty\,\,10\,\,\infty\,\,\infty\,\,\infty]^\T$
to ensure the feasibility (relatively complete recourse) requirements,
as discussed in \cref{sec:fADP}.

Similar to the energy optimization case,
\todob{A reasonable heuristic algorithm would be interesting here too.
       But due to order and fermentation delays, that's not so simple now.}
we use the CAP and AMAP estimators for fADP
with various trajectory set sizes $n$ and cost-to-go evaluation numbers $m$.
As a baseline, we use the optimal solution for the deterministic version
of \eqref{eq:cvx-sto-prg} where
$\vZ_t$ is replaced by its expectation $\E\vZ_t$ for all $t = 1,\ldots,T$.
The results are presented on \cref{fig:result-plan-bb},
\begin{figure}[h!] \begin{center}
  \input{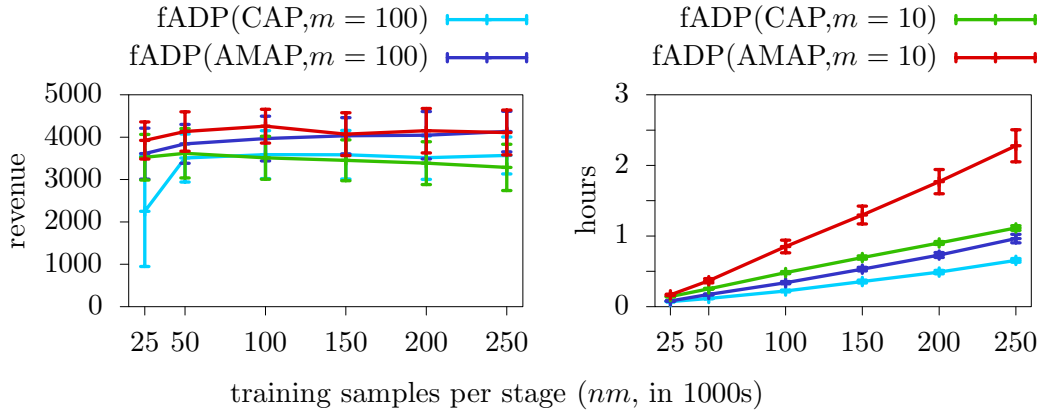}
  \caption{Beer brewery optimization results for fADP algorithm
           using AMAP and CAP convex regression to approximate the
           convex cost-to-go functions.
           Results show the revenue with mean and standard deviation,
           and the training time in minutes for trajectories $n$
           and cost-to-go evaluations $m$.}
  \label{fig:result-plan-bb}
  \vspace{-3mm}
\end{center} \end{figure}
showing that AMAP improves the performance significantly by
collecting revenue over 4100 compared to CAP which stays around 3600.

However, the result also shows that
the running time of AMAP become significantly larger than CAP.
This indicates that larger max-affine models
are trained by AMAP compared to CAP
to improve the accuracy of the cost-to-go approximations,
that increases the computational cost of the LP tasks
of \eqref{eq:sto-prog:yt}, and eventually
slows down the fADP algorithm significantly.

Finally, notice that using larger trajectory sets $n$
for AMAP provide better performance for low sample levels $nm$,
but the improved sample quality with $m = 100$
eventually achieves the same
using significantly less computational resources.
So it seems to be still possible for max-affine estimators
to find better tradeoffs between accuracy and model size.

\section{Conclusions and future work}

In this paper, we considered solving convex SP problems
by the combination of an approximate dynamic programming method (fADP)
and convex regression techniques.
For the latter, we proposed a new state-of-the-art max-affine estimator (AMAP),
and combined it with an approximate dynamic programming algorithm
to address two benchmark convex SP problems of moderate size.

Clearly, scaling up the fADP algorithm for larger problems
would require further improvements.
One of these could be using more expressive convex
piecewise-linear representations
(e.g., sum-max-affine maps),
which might compress the LP tasks enough for the current solvers.
For this, 
\cite{HannahDunson2012} considered various ensemble tehniques
(bagging, smearing, and random forests) to enhance the performance
of the CAP estimator.
However, these techniques still seem to construct too large models
to increase the accuracy significantly,
that makes the vast amount of LP tasks impractical to solve.
Maybe, LP algorithms
which can more efficiently solve large number of similar LP problems
with different right hand sides \citep[e.g.,][]{GassmannWallace1996}
could help with this issue.

But eventually, it would become inevitable
to localize cost-to-go approximations to a fraction of the decision space,
perhaps by running fADP iteratively alternating between
sampling and estimation phases,
and exploring at the boundary of the accurately approximated region
of cost-to-go functions
in order to find and avoid delayed rewards and costs.
However, this is left for future work.

\appendix
\section{Auxiliary tools}

The following result is a slight generalization of
Theorem~5.3 in \citet{Rockafellar1972}.
\begin{lemma}
\label{thm:graph-inf}
  Let $\setX, \setY$ be two convex sets
  and $f : \setX \times \setY$ be a jointly--convex function in its arguments.
  Additionally, let $C : \setX \to 2^\setY$ be a set--valued function
                for which $\graph(C)$ is convex.
  Then $g(\vx) \defeq \inf_{\vy\in C(\vx)}f(\vx,\vy)$ is a convex function.
\end{lemma}
\begin{proof}
  Let $\vx_1, \vx_2 \in \setX$, $\vy_1, \vy_2 \in \setY$
  and $\lambda \in (0,1)$. As $\graph(C)$ is convex,
  \begin{equation*}
    \vy_1 \in C(\vx_1), \vy_2 \in C(\vx_2)
    \quad \Rightarrow \quad
    \lambda\vy_1 + (1-\lambda)\vy_2
      \in C\big(\lambda\vx_1 + (1-\lambda)\vx_2\big)
    \,.
  \end{equation*}
  Using this with the fact that the infimum on a subset becomes larger, 
  and the joint--convexity of $f$, we get
  \begin{equation*} \begin{split}
    g\big(\lambda\vx_1 + (1-\lambda)\vx_2\big)
    &= \inf_{\vz \in C(\lambda\vx_1 + (1-\lambda)\vx_2)}
       f\big(\lambda\vx_1 + (1-\lambda)\vx_2, \vz\big)
    \\
    &\le \inf_{\vy_1 \in C(\vx_1)} \inf_{\vy_2 \in C(\vx_2)}
         f\big(\lambda\vx_1 + (1-\lambda)\vx_2,
               \lambda\vy_1 + (1-\lambda)\vy_2\big)
    \\
    &\le \inf_{\vy_1 \in C(\vx_1)} \inf_{\vy_2 \in C(\vx_2)}
         \lambda f(\vx_1,\vy_1) + (1-\lambda) f(\vx_2,\vy_2)
    \\
    &= \lambda g(\vx_1) + (1-\lambda) g(\vx_2)
    \,,
  \end{split}\end{equation*}
  which proves the convexity of $g$.
\end{proof}

\bibliography{refs}

\begin{thebibliography}{22}
\providecommand{\natexlab}[1]{#1}
\providecommand{\url}[1]{\texttt{#1}}
\expandafter\ifx\csname urlstyle\endcsname\relax
  \providecommand{\doi}[1]{doi: #1}\else
  \providecommand{\doi}{doi: \begingroup \urlstyle{rm}\Url}\fi

\bibitem[Bal\'azs et~al.(2015)Bal\'azs, Gy\"orgy, and
  Szepesv\'ari]{BalazsGyorgySzepesvari2015}
G.~Bal\'azs, A.~Gy\"orgy, and C.~Szepesv\'ari.
\newblock Near-optimal max-affine estimators for convex regression.
\newblock In G.~Lebanon and S.~Vishwanathan, editors, \emph{The 18th
  International Conference on Artificial Intelligence and Statistics
  \mbox{(AISTATS)}}, volume~38 of \emph{JMLR W\&CP}, pages 56--64, 2015.

\bibitem[Bertsekas(2005)]{Bertsekas2005}
D.~P. Bertsekas.
\newblock \emph{Dynamic Programming and Optimal Control, Volume~I}.
\newblock Athena Scientific, 3rd edition, 2005.

\bibitem[Birge and Louveaux(2011)]{BirgeLouveaux2011}
J.~R. Birge and F.~Louveaux.
\newblock \emph{Introduction to Stochastic Programming}.
\newblock Springer, 2011.

\bibitem[Bisschop(2016)]{AIMMS2016}
J.~Bisschop.
\newblock \emph{AIMMS Optimization Modeling}.
\newblock 2016.
\newblock \url{http://www.aimms.com}.

\bibitem[Gassmann and Wallace(1996)]{GassmannWallace1996}
H.~I. Gassmann and S.~W. Wallace.
\newblock Solving linear programs with multiple right-hand sides: Pricing and
  ordering schemes.
\newblock \emph{Annals of Operation Research}, 64:\penalty0 237--259, 1996.

\bibitem[Golub and Loan(1996)]{GolubLoan1996}
G.~H. Golub and C.~F.~V. Loan.
\newblock \emph{Matrix Computations}.
\newblock The Johns Hopkins University Press, 3rd edition, 1996.

\bibitem[Hannah and Dunson(2012)]{HannahDunson2012}
L.~A. Hannah and D.~B. Dunson.
\newblock Ensemble methods for convex regression with applications to geometric
  programming based circuit design.
\newblock In J.~Langford and J.~Pineau, editors, \emph{The 29th International
  Conference on Machine Learning (ICML)}, pages 369--376, 2012.

\bibitem[Hannah and Dunson(2013)]{HannahDunson2013}
L.~A. Hannah and D.~B. Dunson.
\newblock Multivariate convex regression with adaptive partitioning.
\newblock \emph{Journal of Machine Learning Research}, 14:\penalty0 3261--3294,
  2013.

\bibitem[Hannah et~al.(2014)Hannah, Powell, and Dunson]{HannahPowellDunson2014}
L.~A. Hannah, W.~B. Powell, and D.~B. Dunson.
\newblock Semiconvex regression for metamodeling-based optimization.
\newblock \emph{SIAM Journal on Optimization}, 24\penalty0 (2):\penalty0
  573--597, 2014.

\bibitem[Jiang and Powell(2015)]{JiangPowell2015}
D.~R. Jiang and W.~B. Powell.
\newblock An approximate dynamic programming algorithm for monotone value
  functions.
\newblock \emph{CoRR}, 2015.
\newblock \url{http://arxiv.org/abs/1401.1590v6}.

\bibitem[Lov\'asz(1999)]{Lovasz1999}
L.~Lov\'asz.
\newblock Hit-and-run mixes fast.
\newblock \emph{Mathematical Programming}, 86\penalty0 (3):\penalty0 443--461,
  1999.

\bibitem[Magnani and Boyd(2009)]{MagnaniBoyd2009}
A.~Magnani and S.~P. Boyd.
\newblock Convex piecewise-linear fitting.
\newblock \emph{Optimization and Engineering}, 10\penalty0 (1):\penalty0 1--17,
  2009.

\bibitem[Mazumder et~al.(2015)Mazumder, Choudhury, Iyengar, and
  Sen]{MazumderEtAl2015}
R.~Mazumder, A.~Choudhury, G.~Iyengar, and B.~Sen.
\newblock A computational framework for multivariate convex regression and its
  variants.
\newblock 2015.
\newblock \url{http://arxiv.org/abs/1509.08165}.

\bibitem[Powell(2011)]{Powell2011}
W.~B. Powell.
\newblock \emph{Approximate Dynamic Programming: Solving the Curses of
  Dimensionality}.
\newblock John Wiley \& Sons, 2nd edition, 2011.

\bibitem[Puterman(1994)]{Puterman1994}
M.~L. Puterman.
\newblock \emph{Markov Decision Processes: Discrete Stochastic Dynamic
  Programming}.
\newblock John Wiley \& Sons, 1994.

\bibitem[Rockafellar(1972)]{Rockafellar1972}
R.~T. Rockafellar.
\newblock \emph{Convex Analysis}.
\newblock Princeton University Press, 1972.

\bibitem[Ruszczy\'{n}ski and Shapiro(2003)]{RuszczynskiShapiro2003}
A.~P. Ruszczy\'{n}ski and A.~Shapiro.
\newblock \emph{Stochastic Programming}.
\newblock Elsevier, 2003.

\bibitem[Shapiro et~al.(2009)Shapiro, Dentcheva, and
  Ruszczy\'nski]{ShapiroDentchevaRuszczynski2009}
A.~Shapiro, D.~Dentcheva, and A.~Ruszczy\'nski.
\newblock \emph{Lectures on Stochastic Programming, Modeling and Theory}.
\newblock Society for Industrial and Applied Mathematics and the Mathematical
  Programming Society, 2009.

\bibitem[Smith(1984)]{Smith1984}
R.~L. Smith.
\newblock Efficient monte carlo procedures for generating points uniformly
  distributed over bounded regions.
\newblock \emph{Operations Research}, 32\penalty0 (6):\penalty0 1296--1308,
  1984.

\bibitem[Sutton(1998)]{Sutton1998}
R.~S. Sutton.
\newblock \emph{Reinforcement Learning: An Introduction}.
\newblock MIT Press, 1998.

\bibitem[Szepesv\'ari(2010)]{Szepesvari2010}
C.~Szepesv\'ari.
\newblock \emph{Algorithms for Reinforcement Learning: Synthesis Lectures on
  Artificial Intelligence and Machine Learning}.
\newblock Morgan \& Claypool Publishers, 2010.

\bibitem[Vempala(2005)]{Vempala2005}
S.~Vempala.
\newblock Geometric random walk: A survey.
\newblock \emph{Combinatorial and Computational Geometry}, 52:\penalty0
  573--612, 2005.

\end{thebibliography}

\end{document}